\newcommand{\lko}{{ L}_{\kappa\omega}}
\newcommand{\lkk}{{ L}_{\kappa\kappa}}

\renewcommand{\L}{{\cal L}}

\def\bP{{\bf P}}
\def\mm{\mathcal{M}}
\def\mn{\mathcal{N}}

\def\A{\mathcal{A}}

\newenvironment{proof}{{\bf Proof: }}{\ $\Box$}

\newtheorem{theorem}{Theorem}
\newtheorem{definition}[theorem]{Definition}

\newtheorem{claim}{Claim}
\newtheorem{proposition}[theorem]{Proposition}
\newtheorem{corollary}[theorem]{Corollary}
\newtheorem{example}[theorem]{Example}

\def\P{\mathbb{P}}

\def\cN{{\cal N}}
\def\cM{{\cal M}}

\def\lko{L_{\kappa\omega}}

\def\lkk{L_{\kappa\kappa}}

\def\bz{\bar z}
\def\by{\bar y}
\def\bc{\bar c}

\documentclass[12pt]{article}
\usepackage{latexsym}
\usepackage{amsmath}
\usepackage{amssymb,verbatim}
\usepackage{hyperref,proof}
\begin{document}

\title{Positive logics\thanks{     Both authors would like to thank the National Science Foundation no: DMS 1833363. The first author would like to thank the Israel Science Foundation (ISF) grant no: 1838/19 for partial support of this  research and the European Research Council (ERC) advanced grant (Dependent Classes) no: 338821. The second  author would like to thank  the Academy of Finland, grant no: 322795, and funding from the European Research Council (ERC) under the
European Union’s Horizon 2020 research and innovation programme (grant agreement No
101020762).  Publication number 1194. The authors are grateful to Tapani Hyttinen for reading the manuscript and making helpful comments.}}
\author{Saharon Shelah \\
Institute of Mathematics\\ Hebrew University, Jerusalem,
Israel\\
Rutgers University, New Jersey, USA\\
\and Jouko V\"a\"an\"anen\\
Department of Mathematics and Statistics\\ University of Helsinki,
Finland\\
ILLC, University of Amsterdam\\
Amsterdam, Netherlands 
}
\maketitle
\begin{abstract}
Lindstr\"om's Theorem characterizes first order logic as the maximal logic satisfying the Compactness Theorem and the Downward L\"owen\-heim-Skolem Theorem. If we do not assume that logics are closed under negation, there is an obvious extension of first order logic with the two model theoretic properties mentioned, namely existential second order logic. We show  that existential second order logic has a whole family of proper extensions satisfying the Compactness Theorem and the Downward L\"owenheim-Skolem Theorem. Furthermore, we show that in the context of negation-less logics, \emph{positive logics}, as we call them, there is no strongest extension of first order logic with the Compactness Theorem and the Downward L\"owenheim-Skolem Theorem.

\end{abstract}

\section{Introduction}

Our motivating question in this paper is whether we can generalize Lindstr\"om's Theorem from first order logic to $\Sigma^1_1$, that is, existential second order logic. In the case of first order logic Lindstr\"om's Theorem says that first order logic is maximal with the Compactness Theorem\footnote{The $(\kappa,\lambda)$-Compactness Theorem says: Every theory of size $\le\kappa$, every  subset of size $<\lambda$ of which has a model, has a model. Compactness Theorem means  $(\omega,\omega)$-Compactness Theorem.} and the Downward L\"owen\-heim-Skolem Theorem\footnote{The Downward L\"owen\-heim-Skolem Theorem down to $\kappa$ says: Every sentence in a countable vocabulary, which has a model, has a model of size $\le \kappa$. ``Down to $<\kappa$" means ``has a model of size $<\kappa$". The Downward L\"owen\-heim-Skolem Theorem means the  L\"owen\-heim-Skolem Theorem down to $\aleph_0$.} among logics satisfying some minimal closure conditions \cite{MR0244013}. One of the assumed closure conditions is closure under negation. What happens if we drop this assumption?  It seems that this question was first explicitly raised in \cite{MR2134728}. The Compactness Theorem and the Downward L\"owen\-heim-Skolem Theorem make perfect sense, whether we have negation or not. These two conditions make no reference to negation. 

In earlier related work  (\cite{MR2116833})  we 
showed that a strong form of Lindstr\"om's Theorem  fails for
 extensions of \(\lko\) and \(\lkk\):
For weakly compact \(\kappa\) there is no strongest  extension of
\(\lko\) with the 
\((\kappa,\kappa)\)-compactness 
property and the L\"owenheim-Skolem Theorem down to
\(\kappa\). With an additional set-theoretic assumption,
 there is no strongest  extension of
\(\lkk\) with the 
\((\kappa,\kappa)\)-Compactness Theorem and the 
L\"owenheim-Skolem theorem down to \(<\kappa\).

Obviously first order logic itself is not  maximal if negation is dropped because existential second order logic $\Sigma^1_1$, and even $\Sigma^1_{1,\delta}$ (also denoted $PC_\Delta$), i.e. existential second order quantifiers followed by a countable conjunction of first order sentences, which clearly satisfy both the Compactness Theorem and the Downward L\"owen\-heim-Skolem Theorem,  also properly extend first order logic. 

We are  led  to the following (interrelated) questions, all in the context of logics where closure under negation is \emph{not} assumed:
\medskip

\noindent{\bf Question 1:} Is $\Sigma^1_1$ (or rather $\Sigma^1_{1,\delta}$) maximal among logics satisfying the Compactness Theorem and the Downward L\"owen\-heim-Skolem Theorem?
\medskip

\noindent{\bf Question 2:} Is there an extension of $\Sigma^1_1$ (or $\Sigma^1_{1,\delta}$) which is maximal among logics satisfying the Compactness Theorem and the Downward L\"owen\-heim-Skolem Theorem?
\medskip

\noindent{\bf Question 3:} Is there a characterization of $\Sigma^1_1$ (or $\Sigma^1_{1,\delta}$) as maximal among logics satisfying some model-theoretic conditions?
\medskip

\noindent{\bf Question 4:} Is there an extension of $\Sigma^1_1$ (or $\Sigma^1_{1,\delta}$) which is maximal (or even strongest) among logics satisfying some model-theoretic conditions?
\medskip

In this paper we formulate Questions 1 and 2 in exact terms. We answer Question 1 negatively. As to Question 2 we show that there is no strongest\footnote{By strongest extension we mean one which contains every other as a sublogic.} extension of $\Sigma^1_1$ satisfying the Compactness Theorem and the Downward L\"owen\-heim-Skolem Theorem. The existence of a maximal one (which has no proper such extension) remains open. Questions 3 and 4 remain completely unanswered. Admittedly, Question 4 is a little vague as both  ``extension"
and ``model-theoretical conditions" are left open.

To answer the above Questions 1 and 2 we introduce a family of new generalized quantifiers associated with the very natural and intuitive concept of the density of a set of reals. These quantifiers are defined for the purpose of solving the said questions and may lack wider relevance, although the general study of logics without negation is so undeveloped that it may be too early to say what is relevant and what is not.
\medskip

\noindent{\bf Notation:} We use $\mm$ and $\mn$ to denote structures, and $M$ and $N$ to denote their universes, respectively. For finite sequences $s$ and sets $a$, we use $s\char 94\langle a\rangle$ to denote the extension of $s$ by the set $a$. For sequences of length $\le\omega$, $s\triangleleft s'$ means that $s$ is an initial segment of $s'$. The empty sequence is denoted $\emptyset$. A subset $A$ of $2^\omega$ is said to be \emph {dense} if for all $s\in 2^{<\omega}$ there is $s'\in A$ such that $s\triangleleft s'$. We use $\P(\omega)$ to denote the power-set of $\omega$.

\section{Positive logics}
We define the concept of a \emph{positive logic,} meaning a logic without negation, except in front of atomic (and first order) formulas.
We have to be careful about substitution in this context. If we are too lax about the substitution\footnote{The Substitution Property for a abstract logic $L^*$ says that if $\phi$ is in $L^*$, $P$ is an $n$-ary predicate symbol in the vocabulary of $\phi$ and $\psi(x_1,\ldots,x_n)$ is a formula of $L^*$, then the result of substituting $\psi(t_1,\ldots,t_n)$ to occurrences of $P(t_1,\ldots,t_n)$ in $\phi$ is again in $L^*$. For details, see \cite[Def. 1.2.3]{MR819533}.} of formulas into atomic formulas we end up having a logic which is closed under negation, which is not what we want. Substitution is very natural, but it is not needed in Lindstr\"om's characterization of first order logic.

One may ask whether a logic deserves to be called a \emph{logic} if it is not closed under negation?
We do not try to answer this question, but merely point out that there are several logics that do not have a negation in the sense that we have in mind, i.e. in the sense of classical logic. Take, for example, constructive logic. Although it has a negation, it does not have the Law of Excluded Middle, so its negation does not function in the way we mean when we ask whether a logic is closed under negation. In our sense constructive logic is not closed under negation. Another example is continuous logic \cite{MR2436146} and the related positive logic of \cite{MR2371196}. We have already mentioned  existential second order logic $\Sigma^1_1$ and its stronger form, $\Sigma^1_{1,\delta}$. 
In the same category as $\Sigma^1_1$ are Dependence logic \cite{MR2351449} and Independence Friendly Logic \cite{MR2807973}. Transfinite game quantifiers yield infinitary logics which are not closed under negation, due to non-determinacy \cite{MR881268}.  In the finite context there is the complexity class non-deterministic polynomial time NP, which is equivalent to existential second order logic on finite models, of which it is not known whether it is closed under negation. In this paper we introduce new examples of logics without negation.

\begin{definition}
A \emph{positive logic} is an abstract logic\footnote{An abstract logic (or ``a generalized first order logic"), in the sense of \cite{MR0244013} is a pair $L=(\Sigma,T)$, where $\Sigma$ is an arbitrary set and $T$ is a binary relation between members of $\Sigma$ on the one hand and structures on the other. Members of $\Sigma$ are called $L$-sentences. Classes of the form $\{\mm:T(\phi,\mm)\}$, where $\phi$ is an $L$-sentence, are called $L$-characterizable classes. Abstract logics are assumed to satisfy five axioms expressed in terms of $L$-characterizable classes. The axioms correspond to being closed under isomorphism, conjunction, negation, permutation of symbols, and ``free" expansions.} in the sense of \cite{MR0244013} (see also \cite{MR819533}) which contains first order logic and is closed under disjunction, conjunction, and first order quantifiers $\exists$ and $\forall$. We do not require closure under negation, nor closure under substitution. 
\end{definition}

\begin{example}
\begin{enumerate}
\item First order logic  is a positive logic.
\item $\Sigma^1_1$ and $\Sigma^1_{1,\delta}$ are positive logics.
\item If $L$ is a positive logic, then so is $\Sigma^1_1(L)$, the closure of $L$ under existential second order quantification.
\end{enumerate}
\end{example}

\section{A class of new quantifiers}\label{der}

In the tradition of \cite{MR244012} we define our new generalized quantifiers by first specifying a class of structures, closed under isomorphisms.

Let $\tau_d$ be the vocabulary $\{R_0,R_1,R_2,R_3,R_4\}$ consisting of binary predicates $R_0,R_1,R_2$ and unary predicates $R_3,R_4$.

\begin{example}\label{ee}
A canonical example of a $\tau_d$-structure is the model $$\mm_{A}=(M,R^{\mm_A}_0,R^{\mm_A}_1,R^{\mm_A}_2,R^{\mm_A}_3, R^{\mm_A}_4),$$ where $A\subseteq 2^\omega$ and
\begin{itemize}
\item $M=2^{<\omega}\cup A$.
\item  $R_i^{\mm_A}=\{(a,b)\in (2^{<\omega})^2: b=a\char 94\langle i\rangle\}$ ($i=0,1$).
\item  $R_2^{\mm_A}=\{(a,b)\in M\times M: a\triangleleft b\}$.
\item  $R_3^{\mm_A}=\{\emptyset\}.$
\item  $R_4^{\mm_A}=M.$
\end{itemize}\end{example}

\begin{definition}\label{d}
For $n<\omega$ and $\eta\in 2^n$ we define $\psi_\eta(x)$ as:
{\setlength\arraycolsep{1pt}
$$\begin{array}{rl}
R_4(x)\wedge\exists y_0\ldots\exists y_n&(R_3(y_0)\wedge  \bigwedge_{i\le n}R_4(y_i)\wedge \bigwedge_{i<n} y_iR_{\eta(i)}y_{i+1}\wedge\bigwedge_{i\le n}y_i R_2 x).
\end{array}$$}
For a $\tau_d$-model $\cM$ and $a\in M$ we define
 $$\begin{array}{l}
 \Omega({\cM, a})=\{\eta\in 2^{<\omega}: \cM\models\psi_\eta(a)\}\\
\Omega(\cM)=\{\eta\in 2^{\omega} : \mbox{ for some $a\in M$, }
 \eta\restriction n\in \Omega({\cM,a})\mbox{ for all $n<\omega$}\}.
\end{array}$$ If $\eta\in \Omega({\cM, a})$, we say that $a$ \emph{represents} $\eta$ in $\cM$.
We also say that $\cM$ \emph{represents} the set $\Omega(\cM)$.\end{definition}

One element $a$ can represent several $\eta$, but later in Section~\ref{ie} we impose a further restriction to the effect that representation is unique.
 
Note that if $\cM$ is a $\tau_d$-model, then the property ``$\Omega(\cM)$ is dense" is a $\Sigma^1_1$-property of $\cM$.  Since we aim at a logic which goes beyond existential second order logic, we have to sharpen the requirement of density. The property of $\tau_d$-models we are interested in is the property that ``$\Omega(\cM)\setminus A$ is dense" for some preassigned set $A\subseteq 2^\omega$ of reals.

\begin{definition}\label{123}
Let $A\subseteq 2^\omega$. We define the Lindstr\"om quantifier $Q_A$ as follows. Suppose $\mm$ is a model and $\bc\in M^k$. Then we define
\begin{equation}\label{Q}
(Q_Ax_0x_1)(\psi_0(x_0,x_1,\bc),\psi_1(x_0,x_1,\bc),\psi_2(x_0,x_1,\bc),\psi_3(x_0,\bc),\psi_4(x_0,\bc))
\end{equation}to be true in $\mm$
if and only if $\Omega(\mm_{\bar{\psi}})\setminus A$ is dense, where 
$$\bar \psi=(\psi_0(x_0,x_1,\bc),\psi_1(x_0,x_1,\bc),\psi_2(x_0,x_1,\bc),\psi_3(x_0,\bc),\psi_4(x_0,\bc)),$$
$$\mm_{\bar{\psi}}=(M,R^\mn_0,R^\mn_1,R^\mn_2,R^\mn_3, R^\mn_4),$$ and
\begin{itemize}
\item  $R_i^\mn=\{(a,b)\in M^2: \cM\models\psi_i(a,b,\bc)\}$ ($i=0,1,2$).
\item $R_3^\mn=\{a\in M:\cM\models\psi_3(a,\bc)\}$.
\item $R_4^\mn=\{a\in M : \cM\models\psi_4(a,\bc)\}$.

\end{itemize}
\end{definition}

\begin{definition}Suppose $A\subseteq 2^\omega$.
We define the positive logic $L^d_A$ as the closure of first order logic under conjunction, disjunction, first order quantifiers $\exists$ and $\forall$, the existential second order quantifier $\exists R$, where $R$ is a relation symbol, and the generalized quantifier $Q_A$. We denote by $L^{d,\omega}_A$ the extension of $L^d_A$ obtained by allowing countable conjunctions as a logical operation. Finally, the proper class $L^{d,\infty}_A$ denotes the extension of $L^d_A$ obtained by allowing arbitrary set-size conjunctions as a logical operation.
\end{definition}

With the obvious definition of what it means for a positive logic to be a sublogic of another, we can immediately observe that $\Sigma^1_1$ is a sublogic of $L^d_A$, and $\Sigma^1_{1,\delta}$ is a sublogic of $L^{d,\omega}_A$ whatever $A$ is.

\begin{example}\label{e}Suppose $A\subseteq 2^\omega$.
The class $K_A$ of $\tau_d$-models $\mm$ satisfying ``$\hspace{0.5mm}\Omega(\mm)\setminus A$ is dense" is (trivially) definable in $L^d_A$, as $\mm\in K_A$ if and only if $\mm\models\psi_A$, where $$\psi_A=(Q_Ax_0x_1)(R_0(x_0,x_1),R_1(x_0,x_1),R_2(x_0,x_1),R_3(x_0),R_4(x_0)).$$ The model $\mm_B$ of Example~\ref{ee} is in $K_A$, if and only if $B\setminus A$ is dense.
\end{example}

For future reference we make the following observation:
If $\eta\in 2^n$, $\by=(y_0,\ldots,y_{n-1})$,  $\bar \psi$ as in Definition~\ref{123} is a 5-tuple of formulas of $L^d_A$, 
$\bz=(z_0,\ldots,z_{k-1})$, and   $\Gamma^{n,k}_{\bar \psi,\eta}(\by,x,\bz)\in L^d_A$
is the conjunction of 
$$\begin{array}{ll}
\psi_4(y_i,\bz)&\mbox{ for ${i\le n}$}\\
\psi_4(x,\bz)\wedge \psi_3(y_0,\bz)\\
\psi_{\eta(i)}(y_i,y_{i+1},\bz)&\mbox{ for ${i< n}$}\\
\psi_2(y_i,x,\bz)&\mbox{ for ${i\le n}$},\\
\end{array}$$
then (\ref{Q}) is equivalent to 
\begin{equation}\label{QQ}
\begin{array}{l}
\mbox{For every $\sigma\in 2^{<\omega}$ there are $\eta\in 2^\omega\setminus A$ extending $\sigma$} \\
\mbox{and $a\in M$ such that  for some function $n\mapsto \langle b^n_0,\ldots,b^n_{n-1}\rangle$}\\
\mbox{from $\omega$ to $M^n$ we have }\mm\models\Gamma^{n,k}_{\bar \psi,\eta\restriction n}(b^n_0,\ldots,b^n_{n-1},a,\bc)\mbox{ for all $n<\omega$}.
\end{array}
\end{equation}

We proceed to proving that the logic $L^d_A$, for suitably chosen $A\subseteq 2^\omega$, satisfies the Compactness Theorem and the Downward L\"owenheim-Skolem Theorem, and also properly extends $\Sigma^1_{1}$.

\section{The Compactness Theorem}

We use the well-established method of ultraproducts to prove the Compactness Theorem of $L^d_A$.

\begin{theorem}[\L o\' s Lemma for $L^d_A$]\label{Los}Suppose $2^\omega\setminus A$ is dense. Suppose $\cM_i$, $i\in I$, are models and $D$ is an ultrafilter on a set $I$. Let $\cM=\prod_{i\in I}\cM_i/D$, $f_0,\ldots,f_{n-1}\in\prod_{i\in I}M_i$ and $\phi(x_0,\ldots,x_{n-1})$ in $ L^d_A$ (or even in $L^{d,\infty}_A$). Then 
$$\{i\in I : \cM_i\models \phi(f_0(i),\ldots,f_{n-1}(i))\}\in D\Rightarrow
\cM\models \phi(f_0/D,\ldots,f_{n-1}/D).$$
\end{theorem}

\begin{proof}We use induction on $\phi$. The cases corresponding to the atomic formulas, the negated atomic fromulas, conjunction (even infinite conjunction), disjunction, $\exists$, $\forall$, and $\exists R$ (see e.g. \cite[4.1.14]{MR1059055}) are all standard and well known. In the case of disjunction we use the property of ultrafilters that $I_1\cup I_2\in D$ implies $I_1\in D$ or $I_2\in D$. We are left with the induction step for $Q_A$.  Let us denote $f_0(i),\ldots,f_{n-1}(i)$ by $\bar f(i)$ and $f_0/D,\ldots,f_{n-1}/D$ by $\bar f/D$. We   assume
\begin{equation}\label{QA}
J=\{u\in I : \mm_i\models (Q_Ax_0x_1)(\psi_0(x_0,x_1,\bar f(i)),
\ldots,\psi_4(x_0,\bar f(i)))\}\in D
\end{equation}
and demonstrate $\mm\models (Q_A x_0x_1)(\psi_0(x_0,x_1,\bar f/D),
\ldots,\psi_4(x_0,\bar f/D))$.
For $i\in J$ the set $B_i$ of elements $\eta$ of $2^\omega\setminus A$ such that
there are $a_i\in M_i$ and
$b^n_{0,i}\ldots,b^n_{n-1,i}$ in $M_i$ such that  $\mm_i\models\Gamma^{n,k}_{\bar \psi,\eta\restriction n}(b^n_{0,i},\ldots,b^n_{n-1,i},a_i,\bar f(i))$ for all $n<\omega$, is dense.

\medskip

\noindent {\bf Case 1: $D$ is $\aleph_1$-incomplete.}
Let $J=I_0\supseteq I_1\supseteq\ldots$ be a descending chain in $D$ with empty intersection.
We show that the set $B$ of $\eta\in 2^\omega$ such that
there is $a\in M$ such that  for some  $b^n_0,\ldots,b^n_{n-1}$ in $\prod_iM_i/D$ we have $\mm\models\Gamma^{n,k}_{\bar \psi,\eta\restriction n}(b^n_0,\ldots,b^n_{n-1},a,\bar f/D)$ for all $n<\omega$, is the full set $2^\omega$. Since we assume that  $2^\omega\setminus A$ is dense, it follows  that $B\setminus A$ is dense, as we claim. 

Suppose $\eta\in 2^\omega$ is arbitrary. Let $i\in I_{n+1}\setminus I_n$.
Because $B_i$ is dense, there are, for all $n<\omega$, extensions $\eta_i\in 2^\omega$ of $\eta\restriction n$ and elements $a_i,b^n_{i,0},\ldots,b^n_{i,n-1}\in M_i$ such that 
\begin{equation}\label{ss}
\mm_i\models\Gamma^{n,k}_{\bar \psi,\eta_i\restriction n}(b^n_{i,0},\ldots,b^n_{i,n-1},a_i,\bar f(i)) 
\end{equation}
for all $n<\omega$. Let $h(i)=a_i$.
For $i\in I_n\setminus I_{n+1}$ and $m<n$, let $g^n_m(i)=b^n_{i,m}$. Now 
$$\{i\in I : \mm_i\models\Gamma^{n,k}_{\bar \psi,\eta\restriction n}(g^n_0(i),\ldots,g^n_{n-1}(i),h(i),\bar f(i))\}\supseteq I_{n+1}\in D.$$
Hence 
$$\mm\models\Gamma^{n,k}_{\bar \psi,\eta\restriction n}(g^n_0/D,\ldots,g^n_{n-1}/D,h/D,\bar f/D).$$\medskip

\noindent {\bf Case 2: $D$ is $\aleph_1$-complete.}
We show that the set $B$ of $\eta\in 2^\omega\setminus A$ such that
for some $a, b^n_0,\ldots,b^n_{n-1}\in M$ we have $\mm\models\Gamma^{n,k}_{\bar \psi,\eta\restriction n}(b^n_0,\ldots,b^n_{n-1},a,\bar f/D)$ for all $n<\omega$, is dense.
Suppose $\eta\in 2^{n}$.  By the density of $B_i$, for each $i\in J$ there is $\nu_i\in B_i$ extending $\eta$. There is $J_0\subseteq J$ in $D$ such that $\nu_i(n)$ is constant for $i\in J_0$. There is $J_1\subseteq J_0$ in $D$ such that $\nu_i(n+1)$ is constant for $i\in J_{1}$, etc. By $\aleph_1$-completeness we get $J_\infty\in D$ such that $\nu_i(m)$ is constant, say $\eta^*(m)$ for all $m\ge n$ and all  $i\in J_\infty$. Now $\eta^*\in B$ follows easily.
\end{proof}

\begin{corollary}\label{wpp}If $2^\omega\setminus A$ is dense, then:
\begin{enumerate}
\item $L^d_A$ (even $L^{d,\infty}_A$) satisfies the (full) Compactness Theorem.
\item Every sentence of $L^d_A$ with an infinite model has arbitrarily large models. 
\item The only sentences in $L^d_A$ that have a negation (in the usual sense) are the first order (equivalent) ones.\end{enumerate} 
\end{corollary}

\begin{proof} The usual argument gives 1:
Suppose $T$ is a finitely consistent theory in $L^d_A$. Let $I$ be the set of finite subsets of $T$ and for each $i\in I$, let $\mm_i\models i$. If $\phi\in T$, let $A_\phi=\{i\in I : \phi\in i\}$. Then the family ${\cal J}=\{A_\phi : \phi\in T\}$ has the finite intersection property. Let $D$ be a non-principal ultrafilter on $I$ extending $\mathcal{J}$. 
Now if $\phi\in T$, then $\prod_D\mm_i\models\phi$, as
$\{i\in I : \mm_i\models\phi\}\supseteq A_\phi\in D.$

Claim 2 follows immediately from Claim from 1. Claim 3 follows from the ultraproduct characterization of first order model classes (see e.g. \cite[4.1.12]{MR1059055}) and the characterization of elementary equivalence in terms of ultrapowers \cite{MR0297554}. \end{proof}

\begin{theorem}[Robinson's Consistency Lemma for $L^d_A$] Suppose $2^\omega\setminus A$ is dense.
Suppose $T_1$ and $T_2$ are consistent $L^d_A$-theories with vocabularies $\tau_1$ and $\tau_2$, respectively, such that $T_1\cap T_2$ is  complete with respect to first order logic in the vocabulary $\tau_1\cap\tau_2$. Then $T_1\cup T_2$ is consistent.
\end{theorem}

\begin{proof}This proof is not specific to $L^d_A$, but is rather a well-known consequence of \L o\'s Lemma, Theorem~\ref{Los}.
Let $\cM_1\models T_1$ and $\cM_2\models T_2$. Let $\cM_l^-$ be the reduct of $\cM_l$ to the vocabulary $\tau_d=\tau_1\cap\tau_2$. Now $\cM_1^-$ and $\cM_2^-$ are elementarily equivalent in first order logic, for if $\cM_1^-\models\phi$ then necessarily $T_1\cap T_2\models\phi$, whence $\cM_2^-\models\phi$, and vice versa. By \cite{MR0297554} there are a set $I$ and an ultrafilter $D$ on $I$ such that if we denote $\cM_1^I/D$ by $\cN_1$ and $\cM_2^I/D$ by $\cN_2$, then $\cN_1\restriction \tau_d\cong \cN_2\restriction\tau_d$. W.l.o.g. $\cN_1\restriction \tau_d= \cN_2\restriction\tau_d$. Let $\cN$ be a common expansion of $\cN_1$ and $\cN_2$. By Theorem~\ref{Los},  $\cN\restriction\tau_1\models T_1$ and $\cN\restriction\tau_2\models T_2$. Hence, $\cN\models T_1\cup T_2$.
\end{proof}

\section{The Downward L\"owenheim-Skolem property}

The Downward L\"owenheim-Skolem Property, which says that any sentence (of the logic) which has a model has a countable model, is an important ingredient of the Lindstr\"om characterization of first order logic. The main examples of logics with this property, apart from first order logic, are  $L_{\omega_1\omega}$ and its sublogics $L(Q_0)$ (with the quantifier ``there exists infinitely many", see e.g. \cite[p. 8]{zbMATH03941493}) and the weak second order logic $L^2_w$ (with quantifiers for variables that range over finite sets, see e.g. \cite[p. 9]{zbMATH03941493}). We now prove this property for $L^d_A$ in a particularly strong form.
 
Because of lack of negation the elementary submodel relation $\mm\preccurlyeq\mn$ splits into two different concepts $\mm\preccurlyeq^+\mn$ and $\mm\preccurlyeq^-\mn$:

\begin{definition}
Let us write $\mn\preccurlyeq^-_{L^d_A}\mm$ if $\mn\subseteq \mm$ and for all $a_1,\ldots,a_n$ in $N$ and all formulas $\phi(x_1,\ldots,x_n)$ of $L^d_A$ we have
$$\mm\models\phi(a_1,\ldots,a_n)\Rightarrow\mn\models\phi(a_1,\ldots,a_n).$$
Respectively, we write $\mn\preccurlyeq^+_{L^d_A}\mm$ if $\mn\subseteq \mm$ and for all $a_1,\ldots,a_n$ in $N$ and all formulas $\phi(x_1,\ldots,x_n)$ of $L^d_A$ we have
$$\mn\models\phi(a_1,\ldots,a_n)\Rightarrow\mm\models\phi(a_1,\ldots,a_n).$$
Similar definitions can be given for $L^{d,\omega}_A$ and $L^{d,\infty}_A$, and for fragments (i.e. subsets closed under subformulas) $\Gamma$ thereof.
\end{definition}

The Compactness Theorem implies that every infinite model $\mn$ has arbitrarily large $\mm$ such that $\mn\preccurlyeq^+_{L^d_A}\mm$ (Corollary~\ref{wpp}).

\begin{theorem}[Downward L\"owenheim-Skolem-Tarski Theorem]\label{ls}
Sup\-pose $\kappa\ge\aleph_0$, $A\subseteq 2^\omega$, $\mm$ is a model for a vocabulary of cardinality $\le\kappa$, and $X\subseteq M$ such that $|X|\le\kappa$. Then there is $\mn\preccurlyeq^-_{L^d_A}\mm$ (even $\mn\preccurlyeq^-_{\Gamma}\mm$ for any fixed  fragment $\Gamma$ of $L^{d,\infty}_A$ of size $\le\kappa$) such that $X\subseteq N$ and $|N|=\kappa$.
\end{theorem}

\begin{proof}
We first expand $\mm$ as follows: For every $L^d_A$-formula $\phi(R,\bz)$, where $R$ is $n$-ary and $\bz=z_0,\ldots,z_{k-1}$, we make sure there  is a predicate symbol $R^*$ of arity $k+n$ such that if $\mm\models \exists R\phi(R,\bc)$, then $\mm\models\phi(R^*(\bc,\cdot),\bc)$. Likewise, we may assume the vocabulary of $\mm$ has a Skolem function $f_\phi$ for each formula $\phi(x,\bz)$ such that if $\mm\models\exists x\phi(x,\bc)$, then $\mm\models\phi(f_\phi(\bc),\bc)$. Let $\tau$ be the original vocabulary of $\mm$ and $\tau^*$ the vocabulary of the expansion, which we denote $\mm^*$. For any formulas $\bar \psi$ in $L^d_A$ of the vocabulary $\tau^*$ and $\bar c\in M^k$ such that (\ref{Q}) 
in Definition~\ref{123} holds, let  
$g(n,k,\bar \psi, \eta,\bar c)$ be the function which maps $n,k,{\bar \psi},\bar c$ and $\eta\in 2^n$ to $\langle b^n_0,\ldots,b^n_{n-1}\rangle\in M^n$
such that $\mm^*\models\Gamma^{n,k}_{\bar \psi,\eta\restriction n}(b^n_0,\ldots,b^n_{n-1},a,\bc)$ for all $n<\omega$. 
Denoting for any cardinal $\theta$ the set of sets of hereditary cardinality $<\theta$ by $H_\theta$, let $\theta\ge(\kappa+2^\omega)^+$such that $M\subseteq H_\theta$, and $K\prec H_\theta$, such that  $|K|=\kappa$ and $\{A,\kappa,\tau,\mm^*, X,g\}\cup \kappa\cup\tau^*\cup X\subseteq K$. Let $\mn$ be the restriction of $\mm^*$ to $K$, i.e. the universe $N$ of $\mn$ is $M\cap K$ and the constants, relations and functions of $\mm^*$ are relativized to $N$.

We need to check that $N$ is closed under the interpretations of function symbols of the vocabulary of $\mm^*$. Let $f$ be such a function symbol. Suppose $f$ is $s$-ary and $\bc\in N^s$. The sentence $\exists x(x\in M\wedge x=f^{\mm^*}(\bc))$ is true in $H_\theta$, hence true in $K$. Thus there is $b\in N (=M\cap K)$ such that $b=f^{\mm^*}(\bc)$ is true in $K$. Therefore $f^{\mm^*}(\bc)=b\in N$. We can  conclude that $\mn$ is a substructure of $\mm^*$.

\medskip

\noindent{\bf Claim:\ } If $\phi(\vec{x})$ is a $\tau$-formula in $L^d_A$ and $\vec{a}\in N$, then $\mm^*\models\phi(\vec{a})\Rightarrow\mn\models\phi(\vec{a})$. 
\medskip

We use induction on $\phi$. The claim follows from $\mn\subseteq\mm^*$  for  atomic and  negated atomic $\phi$. The claim is clearly preserved under conjunction and disjunction. It is also trivially preserved under universal quantifier, since $\mn\subseteq\mm^*$. The induction steps for both first and second  order existential quantifiers are trivial because of the expansion we have performed on $\mm^*$. We are left with the quantifier $Q_A$.

Suppose $\mm^*$ satisfies (\ref{Q}) of Definition~\ref{123} with $\bc\in N^k$. Thus (\ref{QQ}) holds and we want to prove (\ref{QQ}) with $\mm^*$ replaced by $\mn$. Note that (\ref{QQ}) also holds in $K$. Suppose $\sigma\in 2^n$ is given. There is $\eta\in K\cap (2^\omega\setminus A)$ extending $\sigma$ such that  $K$ satisfies
$$
\begin{array}{l}
\mbox{There is $a\in M$ such that  for some function $n\mapsto \langle b^n_0,\ldots,b^n_{n-1}\rangle$}\\
\mbox{from $\omega$ to $M^n$ we have }\mm^*\models\Gamma^{n,k}_{\bar \psi,\eta\restriction n}(b^n_0,\ldots,b^n_{n-1},a,\bc)\mbox{ for all $n<\omega$}.
\end{array}
$$
Thus there are $a\in N$ and a function $n\mapsto \langle b^n_0,\ldots,b^n_{n-1}\rangle$
from $\omega$ to $N^n$ such that $\mm^*\models\Gamma^{n,k}_{\bar \psi,\eta\restriction n}(b^n_0,\ldots,b^n_{n-1},a,\bc)$ for all $n<\omega$. By the Induction Hypothesis, 
$\mn\models\Gamma^{n,k}_{\bar \psi,\eta\restriction n}(b^n_0,\ldots,b^n_{n-1},a,\bc)$ for all $n<\omega$ follows. \end{proof}

We  conclude that 
every sentence of $L^d_A$ which has an infinite model has a countable model and an uncountable model.

The following examples show that Theorem~\ref{ls} is in a sense optimal:

\begin{example}
There is an uncountable model $\mm$, namely $(\P(\omega),a,\in)$, where $a$ is the element $\omega$ of $\P(\omega)$, such that there is no countable model $\mn$ with 
$\mn\preccurlyeq^+_{\Sigma^1_1}\mm$. There is a countable model $\mn$, namely $(\omega,<)$, such that there is no uncountable model $\mm$ with 
$\mn\preccurlyeq^-_{\Sigma^1_1}\mm$. 
\end{example}

\section{Proper extensions of $\Sigma^1_1$ and $\Sigma^1_{1,\delta}$}\label{s6}

Our goal in this section is to show that for many $A\subseteq 2^\omega$ the logic $L^d_A$ properly extends $\Sigma^1_1$ and $L^{d,\omega}_A$ properly extends $\Sigma^1_{1,\delta}$. We have a spectrum of results to this effect but nothing as conclusive as being able to explicitly point out such a set $A$. There are obvious reasons for this. The logic $\Sigma^1_1$ is very powerful and any ``simple" $A\subseteq 2^\omega$ is likely to yield $L^d_A$ which is  equivalent to $\Sigma^1_1$ rather than properly extending it. This is even more true with $L^{d,\omega}_A$ and $\Sigma^1_{1,\delta}$.

We first establish the basic existence  of sets $A\subset 2^\omega$ with the desired properties. We shall then refine the result with further arguments.

\begin{theorem}\label{nl}
There are sets $A\subseteq 2^\omega$ such that $2^\omega\setminus A$ is dense and $Q_A$ is not definable in $\Sigma^1_1$, nor in $\Sigma^1_{1,\delta}$, nor in $L_{\omega_1\omega_1}$.
\end{theorem}

\begin{proof}
Let $A_\alpha$, $\alpha<2^\omega$, be disjoint dense subsets of $2^\omega$. For any $X\subseteq 2^\omega$, let $$A_X=\bigcup_{\alpha\in X}A_\alpha.$$ Note that if $X\ne Y$, then $A_X\setminus A_Y$ or $A_Y\setminus A_X$ is dense.
 Let $K_A$ and $K_B$ be as in Example~\ref{e} and  $\mm_A$ is as in Example~\ref{ee}.
If $A\setminus B$ is dense, then $K_A\ne K_B$, as $\mm_A\in K_B$ but $\mm_A\notin K_A$.
Thus the classes $K_{A_X}$, $X\subseteq 2^\omega$, are all different. For cardinality reasons there is $X\subseteq 2^\omega$ so that $K_{A_X}$ is not definable in $\Sigma^1_{1,\delta}$, nor in $L_{\omega_1\omega_1}$. But $K_{A_X}$ is always definable in $L^d_{A_X}$.
\end{proof}

The following result merely improves the previous  result:

\begin{theorem}\label{basic}
There is a countable $A_0\subseteq 2^\omega$ such that  if $A_0\subseteq A\subseteq 2^\omega$ with $2^\omega\setminus A$ dense, then $Q_A$ is not $\Sigma^1_1$-definable.
\end{theorem}

\begin{proof}
Let us consider $\mm_{B}$, where $B=2^\omega$ (defined in Example~\ref{ee}). Let $\mm$ be resplendent (see \cite{MR403952}) such that $\mm_B\prec\mm$. Let $\mm^+$ be an expansion of $\mm$ such that every $\Sigma^1_1$-sentence true in $\mm^+$ has a witness in the vocabulary (countable). Let $\mn^+\prec\mm^+$ be countable. Let $A_0$ be the countable set  $\Omega(\mn^+)$. Suppose now
$A_0\subseteq A\subseteq 2^\omega$, but  $\psi_A$, as defined in Example~\ref{e}, is definable by a $\Sigma^1_1$-sentence $\phi$.  Since $2^\omega\setminus A$ dense, $\mm_B\models\psi_A$, whence $\mm_B\models\phi$. Hence all the first order consequences of $\phi$ are true in $\mm_B$. Since $\mm$ is resplendent, $\mm\models\phi$. Since $\phi$ has a witness in $\mm^+$, $\mn^+\models\phi$. Hence  $\mn^+\models\psi_A$ whence $\Omega(\mn^+)\setminus A$ is dense.
This is a contradiction, as $\Omega(\mn^+)\setminus A=\emptyset$.
\end{proof}

\begin{theorem}
 Let $\bP$ be the poset of finite partial functions $(\omega_1+\omega_1)\times \omega\to 2$ i.e. the forcing for adding $\omega_1+\omega_1$ Cohen reals. Let $G$ be $\bP$-generic and $\eta_\alpha\in 2^\omega$, $\alpha<\omega_1+\omega_1$, the Cohen reals added by $G$. Let $A$ be the set of $\eta$ such that $\eta=\eta_\alpha \mbox{ (mod finite)}$ for some $\alpha<\omega_1$. Then in $V[G]$, $Q_A$ is not $\Sigma^1_1$-definable.
\end{theorem}

\begin{proof}
 Let $B$ be the set of $\eta$ such that $\eta=\eta_\alpha (\mbox{mod finite})$ for some $\alpha<\omega_1+\omega_1$. Then $\mm_B\models\psi_A$. Suppose $\phi$ is a $\Sigma^1_1$-sentence logically equivalent to $\psi_A$. Thus $\mm_B\models\phi$. Let $f$ be a bijection (in $V$) of $\omega_1+\omega_1$ onto $\omega_1$. The function $f$ induces an complete embedding $\bar f$ of $\bP$ into $\bP$.  The mapping $\bar f$ induces a mapping $\tau\mapsto\tau_{\bar f}$ between $\bP$-terms. Let $\mn$ be the image of $\mm_B$ under this mapping. Now $\mn\not\models\psi_A$. However, $\mn\models\phi$, whence $\mn\models\psi_A$, a contradiction.
\end{proof}


\medskip

\begin{theorem}\label{wp}
Assume $ A=2^\omega\setminus D$, where $D \subseteq 2^\omega $ is  dense, $\omega<|D|< 2 ^ {\aleph_0} $ and there is an open set $U$ such that $D\cap V$ is uncountable for every non-empty open $V\subseteq U$. Then the 
quantifier $ Q_ A $ is not $ \Sigma ^1_1 $-definable.






\end{theorem} 

\begin{proof}
Recall that $\psi_A$ is a sentence of $L^d_A$ in the vocabulary  $\tau_d$ saying that $\Omega(M)\setminus A$ is dense.
 Thus $\psi_A$ says $\Omega(M)\cap D$ is dense.
Let $ \phi $ be a $\Sigma^1_1$ sentence $\exists R\phi_0$ such that $\psi_A$ and $\phi$ are logically equivalent, contradicting our desired
conclusion.
%
%
Let $ \langle D_ \alpha : \alpha < \omega _1\rangle $ be a sequence of disjoint countable dense subsets of $ D \cap U$.
Let  $ \mn_ \alpha $ be a countable model representing the set $ D_ \alpha$, whence it satisfies $ \psi_A $, hence $\phi$, and there is an expansion $ \mn^*_ \alpha $ of $\mn_
\alpha $ to a model of $ \phi_0$. Let $ \mathcal{N}=\langle N^*_ \alpha : \alpha < \omega _1 \rangle $.
Let $ \mathcal{B}=(H_\theta, \in, <) $, for a large enough cardinal $\theta$ and for 
 a well-ordering  $<$ of $H_\theta$.
%
We choose a countable elementary submodel 
$ \mathcal{B}^* $ of $ \mathcal{B}$ such that $\{\mathcal{N},A,\omega_1\}\subset \mathcal{B}^*$. 

By Theorem IV.5.19 of 
\cite{MR1083551} there is a sequence 
$ \langle  \mathcal{B} _ \alpha : \alpha< 2^{\omega}\rangle $ of countable elementary extensions of $ \mathcal{B}^* $ such that for every $\alpha<\beta< 2^\omega$:
\begin{description}
\item [(a)] $\mathcal{B} _ \alpha$ has standard $ \omega$.
\item [(b)] $\mathcal{B} _ \alpha$ has a (possibly non-standard) member  $ c_ \alpha $ of $(\omega _1)^{\mathcal{B}^* } $.
\item [(c)] If an element of $ {}^{ \omega } 2 $ is definable  in both $\mathcal{B} _ \alpha$ and  $\mathcal{B} _ {\beta}$, 
then it is in  $ \mathcal{B}^*$.
\end{description}
%

Let $ \mn_ \eta ^+ $ 
be the $ c_ \eta$'th
member of the sequence $ \langle \mn^*_ \alpha : \alpha < \omega _1 \rangle $
as interpreted in $ \mathcal{B} _ \eta $.
So necessarily $\mn_ \eta ^+$ is a model of $ \phi_0 $ and hence its reduct 
$\mn_ \eta ^+\restriction\tau_d$  is a model of $\phi$, and further of $\psi_A$.
%
We have continuum many models $\mn_ \eta ^+\restriction\tau_d$ of $\psi_A$.
However, we will now show that the number 
of $ \eta $ for which the model $\mn_ \eta ^+\restriction\tau_d$ satisfies $\psi_A$ is at most 
$|D|<2^\omega$, a contradiction. Suppose $\mn_\eta^+\restriction\tau_d\models \psi_A$. Then the subset of $2^\omega$ represented by $\mn_\eta^+=(\mn^*_{c_\eta})^{\mathcal{B}_\eta}$, i.e. $(D_{c_\eta})^{\mathcal{B}_\eta}$, meets $D$ in a dense set.  Every element of $(D_{c_\eta})^{\mathcal{B}_\eta}$ is definable in $\mathcal{B} _ \eta$.
By the disjointness clause (c)  above we get the claimed contradiction.

We now finish the proof of Theorem~\ref{wp}: Suppose $\eta$ is such that $\mn_\eta^+\not\models \psi_A$. This is a contradiction because  $\mn_\eta^+\models\phi$.
\end{proof}

\section{No strongest extension}\label{ie}

We show that there is no strongest extension among positive logics of first order logic, or $\Sigma^1_1$, or $\Sigma^1_{1,\delta}$, with the Compactness Theorem and the Downward L\"owenheim-Skolem Theorem.

We consider sequences $\A=\langle A_\alpha:\alpha\le\omega_1\rangle$ such that each $A_\alpha$, $\alpha<\omega_1$, is a countable dense subset of $2^\omega$, $\alpha<\beta$ implies $A_\alpha\subset A_\beta$, $A_{\omega_1}=\bigcup_{\alpha<\omega_1}A_\alpha$, and the set $S=\{\alpha<\omega_1 : A_\alpha=\bigcup_{\beta<\alpha}A_\beta\}$ is stationary.

Let $\Theta_{\rm TL}$ be  the first order sentence
$$\begin{array}{l}
\exists x(R_3(x)\wedge\forall y(\neg R_4(y)\vee R_2(x,y)))\wedge\\
\forall x\forall y(\neg R_0(x,y)\vee R_2(x,y))\wedge\\
\forall x\forall y(\neg R_1(x,y)\vee R_2(x,y))\wedge\\
\forall x\forall y(\neg R_2(x,y)\vee (R_4(x)\wedge R_4(y)))\wedge\\
\forall x (\neg R_4(x)\vee R_2(x,x))\wedge\\
\forall x\forall y\forall z(\neg R_2(x,y)\vee \neg  R_2(y,z)\vee R_2(x,z))\wedge\\
\forall x\forall y\forall z(\neg R_2(y,x)\vee\neg R_2(z,x)\vee R_2(y,z)\vee R_2(z,y)).
\end{array}$$
Intuitively, $\Theta_{\rm TL}$ says that $R_2$ is a tree-like partial order extending $R_0$ and $R_1$. For example, the model $\mm_A$ of Example~\ref{ee} always satisfies $\Theta_{\rm TL}$.
If $\mm\models\Theta_{\rm TL}$, then one element $a$ of $M$ can represent only one $\eta$, i.e.
\begin{equation}\label{unique}
\mbox{$\eta,\eta'\in \Omega(\mm,a)$ implies $\eta=\eta'$.}
\end{equation} 

\begin{definition}
We define the Lindstr\"om quantifier $Q_\A$ as follows. Suppose $\mm$ is a model and $\bc\in M^k$. Then we define that $\mm$ satisfies
\begin{equation}\label{Qcal}
(Q_\A x_0x_1)(\psi_0(x_0,x_1,\bc),\psi_1(x_0,x_1,\bc),\psi_2(x_0,x_1,\bc),\psi_3(x_0,\bc),\psi_4(x_0,\bc))
\end{equation}
if and only if $\mm_{\bar{\psi}}\models \Theta_{\rm TL}$ and $\Omega(\mm_{\bar{\psi}})\cap A_{\omega_1}\in \A$, where $\mm_{\bar{\psi}}$ is as in Definition~\ref{123} and 
$\Omega(\mm_{\bar{\psi}})$ is as in Definition~\ref{d}.\end{definition}

\begin{definition}\label{19}
We define $L^d_\A$ as the closure of first order logic under $\wedge,\vee,\exists$,
$\forall,\exists R$ and $Q_\A$. The fragment, where $Q_\A$ is applied to first order formulas $\bar \psi$ only is denoted $L^{d^-}_\A$. Similarly, $L^{d,\omega}_\A$, $L^{d,\infty}_\A$, 
$L^{d^-,\omega}_\A$, and $L^{d^-,\infty}_\A$. 
\end{definition}

\begin{theorem}[\L o\' s Lemma for $L^d_\A$]\label{Loss}Suppose $\cM_i$, $i\in I$, are models and $D$ is an $\omega_1$-incomplete ultrafilter on a set $I$. Let $\cM=\prod_{i\in I}\cM_i/D$, $f_0,\ldots,f_{n-1}\in\prod_{i\in I}M_i$ and $\phi(x_0,\ldots,x_{n-1})$ in  $L^d_\A$ (even in $L^{d,\infty}_\A$). Then 
$$\{i\in I : \cM_i\models \phi(f_0(i),\ldots,f_{n-1}(i))\}\in D\Rightarrow
\cM\models \phi(f_0/D,\ldots,f_{n-1}/D).$$
\end{theorem}

\begin{proof}We follow the proof of Theorem~\ref{Los}. The only point that requires attention is  the induction step for $Q_\A$. We   assume
\begin{equation}\label{QsA}
J=\{u\in I : \mm_i\models Q_\A x_0x_1\psi_0(x_0,x_1,\bar f(i))
\ldots\psi_4(x_0,\bar f(i))\}\in D
\end{equation}
and demonstrate $M\models Q_\A x_0x_1\psi_0(x_0,x_1,\bar f/D)
\ldots\psi_4(x_0,\bar f/D)$.
As in the proof of Theorem~\ref{Los}, it can be shown that the set $B$ of $\eta\in 2^\omega$ such that
there is $a\in M$ such that  for some  
$ b^n_0,\ldots,b^n_{n-1}$ in $\prod_iM_i/D$ we have $\mm\models\Gamma^{n,k}_{\bar \psi,\eta\restriction n}(b^n_0,\ldots,b^n_{n-1},a,\bar f/D)$ for all $n<\omega$, is the full set $2^\omega$. It follows that  $2^\omega\cap A_{\omega_1}=A_{\omega_1}$ and hence that  $2^\omega\cap A_{\omega_1}\in\A$, as  claimed. 
\end{proof}

\begin{corollary}If $2^\omega\setminus A$ is dense, then
$L^d_\A$ (even $L^{d,\infty}_\A$) satisfies the (full) Compactness Theorem.
\end{corollary}

\begin{proof}
The ultrafilter we used in the proof of Corollary~\ref{wpp} was regular, hence $\omega_1$-incomplete.
\end{proof}

We can prove the Downward L\"owenheim-Skolem-Tarski Theorem for $L^{d^-}_\A$ only (see Proposition~\ref{nols} and Theorem~\ref{yes}). 

\begin{theorem}[Downward L\"owenheim-Skolem-Tarski Theorem]\label{lsk}
Sup\-pose $\mm$ is a model for a countable vocabulary and $X\subseteq M$ is countable. Then there is $\mn\preccurlyeq^-_{L^{d^-}_\A}\mm$ (even $\mn\preccurlyeq^-_{\Gamma}\mm$ for any fixed countable fragment of $L^{d,\omega}_\A$) such that $X\subseteq N$ and $|N|\le\aleph_0$.
\end{theorem}

\begin{proof}
We first expand $\mm$ as follows: For every $L^{d^-}_\A$-formula $\phi(R,\bz)$, where $R$ is $n$-ary and $\bz=z_0,\ldots,z_{k-1}$, there is a predicate symbol $R^*$ of arity $k+n$ such that if $\mm\models \exists R\phi(R,\bc)$, then $\mm\models\phi(R^*(\bc,\cdot),\bc)$. Likewise, we may assume the vocabulary of $\mm$ has a Skolem function $f_\phi$ for each formula $\phi(x,\bz)$ such that if $\mm\models\exists\phi(x,\bc)$, then $\mm\models\phi(f_\phi(\bc),\bc)$. Let $\tau$ be the original vocabulary of $\mm$ and $\tau^*$ the vocabulary of the expansion, which we also denote $\mm$. For any formulas $\bar \psi$ in $L^{d-}_\A$ of the vocabulary $\tau^*$ let   
$g(n,k,\bar \psi, \eta)$ be the function which maps $n,k,{\bar \psi}$ and $\eta\in 2^n$ to $\Gamma^{n,k}_{\bar \psi,\eta}(y_0,\ldots,y_n,x,\bz)$. Recall that $S=\{\alpha<\omega_1 : A_\alpha=\bigcup_{\beta<\alpha}A_\beta\}$ is stationary. Let $K\prec H_\theta$, where $\theta\ge(2^\omega)^+$ such that $M\subseteq H_\theta$, $|K|=\aleph_0$, $\{\A,\omega_1,\tau^*,\mm, X,g\}\cup \omega_1\cup\tau\cup X\subseteq K$, and $\delta=K\cap\omega_1\in S$.
Let $\mn$ be the restriction of $\mm$ to $K$, i.e. the universe $N$ of $\mn$ is $M\cap K$ and the constants, relations and functions of $\mm$ are relativized to $N$.

As in the proof of Theorem~\ref{ls},  $N$ is closed under the interpretations of function symbols of the vocabulary of $\mm$. 

\medskip

\noindent{\bf Claim:\ } If $\phi(\vec{x})$ is a $\tau^*$-formula in $L^{d^-}_\A$, then $\mm\models\phi(\vec{c})\Rightarrow\mn\models\phi(\vec{c})$. 
\medskip

We use induction on $\phi$.  In light of the proof of Theorem~\ref{ls}, we only need to consider the quantifier $Q_\A$.
Suppose $\mm$ satisfies (\ref{Qcal}) with $\bc\in N^k$. 
Thus $\Omega(\mm_{\bar{\psi}})\cap A_{\omega_1}\in \A$. Hence
$$K\models``\Omega(\mm_{\bar{\psi}})\cap A_{\omega_1}\in \A".$$ Note that since $\delta\in S$, $K\cap A_{\omega_1}=A_\delta.$

\noindent{\bf Case 1:} $\Omega(\mm_{\bar{\psi}})\cap A_{\omega_1}=A_\alpha$ for some $\alpha<\omega_1$. Then $\alpha<\delta$ and  
$$K\models``\Omega(\mm_{\bar{\psi}})\cap A_{\omega_1}=A_\alpha".$$ %
We prove $\Omega(\mn_{\bar{\psi}})\cap A_{\omega_1}=A_\alpha,$
from which $\Omega(\mn_{\bar{\psi}})\cap A_{\omega_1}\in\A$ follows.

Let first $\eta\in \Omega(\mn_{\bar{\psi}})\cap A_{\omega_1}$. 
There are  $a\in N$ and $ b^n_0,\ldots,b^n_{n-1}$ in  $N$ such that $\mn\models\Gamma^{n,k}_{\bar \psi,\eta\restriction n}(b^n_0,\ldots,b^n_{n-1},a,\bc)$ for all $n<\omega$.
Since the formulas of $\bar \psi$ are first order, $\mm\models\Gamma^{n,k}_{\bar \psi,\eta\restriction n}(b^n_0,\ldots,b^n_{n-1},a,\bc)$ for all $n<\omega$. Hence $\eta\in
\Omega(\mm_{\bar{\psi}})\cap A_{\omega_1}=A_\alpha.$

For the converse, let $\eta\in A_\alpha$. Note that now $\eta\in K$. By the choice of $\alpha$,
$\eta\in
\Omega(\mm_{\bar{\psi}})$. Hence there is $a\in M$ such that  for some  $ b^n_0,\ldots,b^n_{n-1}$
in  $M$ we have $\mm\models\Gamma^{n,k}_{\bar \psi,\eta\restriction n}(b^n_0,\ldots,b^n_{n-1},a,\bc)$ for all $n<\omega$. Such an $a$ and such $b^n_0,\ldots,b^n_{n-1}$ exist also in $K$, by elementarity, as $\eta\in K$.  
By Induction Hypothesis, $\mn\models\Gamma^{n,k}_{\bar \psi,\eta\restriction n}(b^n_0,\ldots,b^n_{n-1},a,\bc)$ for all $n<\omega$. Thus $\eta\in
\Omega(\mn_{\bar{\psi}})$.
\smallskip

\noindent{\bf Case 2:} $\Omega(\mm_{\bar{\psi}})\cap A_{\omega_1}=A_{\omega_1}$. Then   
$K\models``\Omega(\mm_{\bar{\psi}})\cap A_{\omega_1}=A_\delta".$ %
We prove $\Omega(\mn_{\bar{\psi}})\cap A_{\omega_1}=A_\delta,$
from which $\Omega(\mn_{\bar{\psi}})\cap A_{\omega_1}\in\A$ follows.

Let first $\eta\in \Omega(\mn_{\bar{\psi}})\cap A_{\omega_1}$. As in Case 1,
 $\eta\in
\Omega(\mm_{\bar{\psi}})$. Because we have (\ref{unique}), that is, $\eta$ is determined by an element of $N$, we may conclude $\eta\in K$. By $K\prec H_\theta$, $\eta\in (A_{\omega_1})^K$. Hence $\eta\in A_\delta$.

For the converse, let $\eta\in A_\delta$. Since $\delta\in S$, $A_\delta\subset K$,  and hence $\eta\in K$. On the other hand,  $\eta\in
\Omega(\mm_{\bar{\psi}})$, since $A_\delta\subset A_{\omega_1}=\Omega(\mm_{\bar{\psi}})\cap A_{\omega_1}$. Now we can argue as in Case 1 to conclude   $\eta\in
\Omega(\mn_{\bar{\psi}})$.
\end{proof}
\medskip

A consequence of Corollary~\ref{Loss} and Theorem~\ref{lsk} is that the positive logic $L^{d^-}_\A$ is an extension of $\Sigma^1_1$ with both the Compactness Theorem and the Downward L\"owenheim-Skolem Theorem. Similarly, $L^{d^-,\omega}_\A$ is such an extension of $\Sigma^1_{1,\delta}$. 

\begin{theorem}\label{main}There are positive logics $L_1$ and $L_2$ such that 
\begin{enumerate}
\item $L_1,L_2$ both (properly) extend $\Sigma^1_1$.
\item $L_1,L_2$ both satisfy the Compactness Theorem and the Downward L\"owenheim-Skolem Theorem.
\item There is no logic $L_3$ such that $L_1\le L_3$, $L_2\le L_3$, and $L_3$ satisfies
the Downward L\"owenheim-Skolem Theorem.
\end{enumerate}
We can replace $\Sigma^1_1$ by $\Sigma^1_{1,\delta}$.
\end{theorem}

\begin{proof}
Let $\A$ be as above but $A_\alpha=\bigcup_{\beta<\alpha}A_\beta$ for all limit $\alpha$. Let $S,S'\subseteq\omega_1$ be disjoint stationary sets. Note that the set of elements of $S$ that are limits of elements of $S$ is stationary, because it contains the  intersection of $S$ with the closed unbounded set of limits of elements of $S$. Similarly,  the set of elements of $S'$ that are limits of elements of $S'$ is stationary. Let $\A=\langle A_\alpha:\alpha\in S\rangle\char 94\langle A_{\omega_1}\rangle$ and $\A'=\langle A_\alpha:\alpha\in S'\rangle\char 94\langle A_{\omega_1}\rangle$. Now both $\{\alpha\in S:A_\alpha=\bigcup_{\beta\in \alpha\cap S}A_\beta\}$ and $\{\alpha\in S':A_\alpha=\bigcup_{\beta\in \alpha\cap S'}A_\beta\}$ are stationary.
Let 
$$\psi_{\A}=(Q_{\A}x_0x_1)(R_0(x_0,x_1),R_1(x_0,x_1),R_2(x_0,x_1),R_3(x_0),R_4(x_0))$$
and similarly $\psi_{\A'}$.
Let $\phi$ be the sentence $\psi_{\A}\wedge \psi_{\A'}\wedge\Theta_{TL}$. This sentence has a model, namely $\mm_{A_{\omega_1}}$. Suppose it has a countable model $\mn$. Then $\Omega(\mn)\cap A_{\omega_1}\in \A\cap\A'$. Hence $\Omega(\mn)\cap A_{\omega_1}= A_{\omega_1}$. Since $\mn\models\Theta_{TL}$, $N$ must be uncountable, a contradiction.  
\end{proof}

\begin{corollary}
No extension of $\Sigma_1^1$ is strongest with respect to the Compactness Theorem and the Downward L\"owenheim-Skolem Theorem,  among positive logics.
\end{corollary}

We shall now prove that Theorem~\ref{lsk} does not hold with $L^{d-}_\A$ replaced by $L^d_\A$.

\begin{proposition}\label{nols}
Suppose $\A$ is as above. There is an uncountable model $\mm$  for a countable vocabulary such that  there is no countable $\mn\preccurlyeq^-_{L^{d}_\A}\mm$.
\end{proposition} 

\begin{proof}
Let $M$ be the union of ${2}^\omega$, ${2}^{<\omega}$ and ${2}^\omega \times\omega_1$. The relations of the structure $\mm$ are
\begin{enumerate}
\item  $R_i^{\mm}=\{a\char 94\langle i\rangle:a\in {2}^{<\omega}\}$ ($i=0,1$).
\item  $R_2^{\mm}=\{(a,b)\in (2^{<\omega})\times 2^\omega: a\triangleleft b\}$.
\item  $R_3^{\mm}=\{\emptyset\}.$
\item  $R_4^{\mm}={2}^\omega\cup {2}^{<\omega}.$
\item  $R_5^{\mm}=\{(a,(a,\alpha)): a\in 2^\omega, \alpha<\omega_1\}.$
\item  $R_6^{\mm}={2}^\omega.$
\item  $R_7^{\mm}={2}^{<\omega}.$
\item  $Q_1^{\mm}={2}^{\omega}\times\omega_1.$
\item  $Q_2^{\mm}=\{(a,\alpha)\in Q_1: (a\in A_0\wedge\alpha<\omega)\vee
(a\in A_2\setminus A_1\wedge\alpha<\omega_1)\}.$
\item $Q_3^{\mm}=\{(a,b) : \exists\alpha(a\in A_\alpha\wedge b\in A_{\omega_1}\setminus A_\alpha)\}.$
\end{enumerate} 
Suppose $\mn\preccurlyeq^-_{L^{d}_\A}\mm$ is countable. Let $\phi(x)$ be the existential second order formula
$$R_6(x)\wedge\exists F(\mbox{$F$ is a one-one function from $R_7$ onto $\{y:Q_2(x,y)\}$}).$$

\begin{itemize}
\item $\mm\models\phi(a)$ if and only if $a\in A_0^{\mm}$.
\item $\mn\models\phi(a)$ if and only if $a\in A_0^{\mn}\cup (A_2^{\mn}\setminus A_1^{\mn})$.
\end{itemize}

Thus   $$\mm\models (Q_\A x_0x_1)(\psi_0(x_0,x_1,\bc),\psi_1(x_0,x_1,\bc),\psi_2(x_0,x_1,\bc),\psi_3(x_0,\bc),\psi_4(x_0,\bc))$$
but 
 $$\mn\not\models (Q_\A x_0x_1)(\psi_0(x_0,x_1,\bc),\psi_1(x_0,x_1,\bc),\psi_2(x_0,x_1,\bc),\psi_3(x_0,\bc),\psi_4(x_0,\bc))$$
\end{proof}

Despite the negative result of Theorem~\ref{nols}, Theorem~\ref{lsk} still holds for the fragment of  $L^d_\A$ obtained by dropping existential second order quantifiers.

\begin{definition}
Let $L^{d0}_\A$ be defined as $L^d_\A$ (Definition~\ref{19}) except that existential second order quantification is not allowed. Let $L^{d1}_\A$ be defined as the extension of $L^d_\A$ by adding negation to the logical operations.
\end{definition}

Clearly, $L^{d0}_\A$ is a positive logic and it satisfies the Compactness Theorem because even $L^d_\A$ does. The logic $L^{d1}_\A$ is an abstract logic in the sense of \cite{MR0244013}. Unlike our positive logics, it is closed under negation and also closed under substitution. Note that $L^{d0}_\A\le L^{d1}_\A$. 

\begin{theorem}[Downward L\"owenheim-Skolem-Tarski Theorem]\label{yes}
Sup\-pose $\mm$ is a model for a countable vocabulary and $X\subseteq M$ is countable. Then there is $\mn\preccurlyeq_{L^{d1}_\A}\mm$  such that $X\subseteq N$ and $|N|\le\aleph_0$.
In particular,  $\mn\preccurlyeq_{L^{d0}_\A}\mm$.
\end{theorem}

\begin{proof} This is as in the proof of Theorem~\ref{lsk}.
We first expand $\mm$ as follows: For every $L^{d1}_\A$-formula $\phi(\bz)$, where  $\bz=z_0,\ldots,z_{k-1}$, there is a predicate symbol $R_\phi$ of arity $k$ such that  $\mm\models \forall\bz(\phi(\bz)\leftrightarrow R_\phi(\bz))$. Let $\tau$ be the original vocabulary of $\mm$ and $\tau^*$ the vocabulary of the expansion. For any atomic formulas $\bar \psi$  of the vocabulary $\tau^*$ let     
$g(n,k,\bar \psi, \eta)$ be the function which maps $n,k,{\bar \psi}$ and $\eta\in 2^n$ to $\Gamma^{n,k}_{\bar \psi,\eta}(y_0,\ldots,y_n,x,\bz)$. Let $K\prec H_\theta$, where $\theta\ge(2^\omega)^+$ such that $M\subseteq H_\theta$, $|K|=\aleph_0$, $\{\A,\omega_1,\tau^*,\mm, X,g\}\cup \omega_1\cup\tau\cup X\subseteq K$, and $\delta=K\cap\omega_1\in S$. Let $\mn$ be the restriction of $\mm$ to $K$, i.e. the universe $N$ of $\mn$ is $M\cap K$ and the constants, relations and functions of $\mm$ are relativized to $N$.

As in the proof of Theorem~\ref{ls},  $N$ is closed under the interpretations of function symbols of the vocabulary of $\mm$. 

\medskip

\noindent{\bf Claim:\ } If $\phi(\vec{x})$ is a $\tau^*$-formula in $L^{d1}_\A$ and $\vec{c}\in N$, then $\mn\models\phi(\vec{c})\leftrightarrow R_\phi(\vec{c})$. 
\medskip

The proof of this claim is as in the proof of Theorem~\ref{lsk}. Since $\mn\subseteq\mm$ in the vocabulary $\tau^*$, the claim implies $\mn\preccurlyeq_{L^{d1}_\A}\mm$.
\end{proof}

The logic $L^{d1}_\A$ is closed under negation and satisfies the Downward L\"owen\-heim-Skolem Theorem. Thus it cannot satisfy the Compactness Theorem, although its sublogic $L^{d0}_\A$ does.

\bibliographystyle{plain}
\bibliography{positive}
\end{document}